\documentclass{amsart}

\usepackage{amssymb}
\usepackage[british]{babel}

\newcommand{\fa}[2]{{#1}\langle{#2}\rangle}
\newcommand{\N}{\mathbb{N}}
\newcommand{\Z}{\mathbb{Z}}
\newcommand{\Q}{\mathbb{Q}}
\newcommand{\field}{\Bbbk}

\theoremstyle{plain}
\newtheorem{theorem}{Theorem}
\newtheorem{proposition}[theorem]{Proposition}
\newtheorem{lemma}[theorem]{Lemma}
\newtheorem{corollary}[theorem]{Corollary}

\theoremstyle{remark}
\newtheorem*{remark}{Remark}

\DeclareMathOperator{\End}{End}
\DeclareMathOperator{\trace}{tr}
\DeclareMathOperator{\imag}{im}
\DeclareMathOperator{\charac}{char}

\begin{document}

\title{Rationality of the Hilbert series of Hopf-invariants of free algebras}

\author{Vitor O.~Ferreira}
\thanks{The first author was partially supported by
  CNPq (Grant 308163/2007-9) and by FAPESP (Projeto Tem\'atico 2009/52665-0).}
  
\author{Lucia S.~I.~Murakami}
\thanks{The second author was partially supported by FAPESP (Projeto Tem\'atico 2010/50347-9).}

\address{Department of Mathematics, Institute of Mathematics
and Statistics, University of S\~ao Paulo, Caixa Postal 66281,
S\~ao Paulo - SP, 05314-970, Brazil}

\email{vofer@ime.usp.br, ikemoto@ime.usp.br}

\date{26 May 2011}

\subjclass[2010]{16S10, 16T05, 16W22,15A72}

\keywords{Free associative algebra; Hopf algebra actions; invariants}

\begin{abstract}
  It is shown that the subalgebra of invariants of a free associative
  algebra of finite rank under a linear action of a semisimple Hopf
  algebra has a rational Hilbert series with respect to the usual
  degree function, whenever the ground field has zero characteristic.
\end{abstract}

\maketitle

\section*{Introduction}

The results contained in this paper contribute
to the study of the structure of free algebras under linear actions
of Hopf algebras that the authors have been pursuing recently. For
instance, it has been proved in \cite{FMP04} that the invariants of a free 
associative algebra $\fa{\field}{X}$ under a linear action of a Hopf algebra $H$ form a 
free subalgebra. Moreover, if $H$ finite-dimensional and pointed, then,
in the same paper, a Galois correspondence between right coideal subalgebras
of $H$ and free subalgebras of $\fa{\field}{X}$ containing the subalgebra $\fa{\field}{X}^H$
of invariants of $\fa{\field}{X}$ under $H$ has been described. 
Additionally, if $H$ is a finite-dimensional Hopf algebra 
which is generated by grouplikes and skew-primitives, then in \cite{FM07}
it is shown that the free subalgebra $\fa{\field}{X}^H$ of $\fa{\field}{X}$
is finitely generated only when the action of $H$ is scalar.

The object under analysis in this note is the Hilbert series of the subalgebra
of invariants $\fa{\field}{X}^H$, for $H$ a semisimple Hopf algebra. 
Since a linear action $H$ on $\fa{\field}{X}$ is homogeneous, the subalgebra 
$\fa{\field}{X}^H$
inherits the grading of $\fa{\field}{X}$ by the usual degree. Therefore,
if $X$ is finite, the Hilbert series of $\fa{\field}{X}$ with respect to
this grading is defined. We show that for a field $\field$ with $\charac{\field}=0$,
this Hilbert series is a rational function. The proof of this fact provides
an explicit formula for this rational function in terms of characters of
the action of $H$ on the subspace generated by $X$.
We are eventually able to obtain Dicks-Formanek's formula for the
case of finite groups of automorphisms of $\fa{\field}{X}$ in \cite{DF82}
as a special case.

\medskip

Hereafter vector spaces and algebras will be defined over a field $\field$;
tensor products are taken over $\field$. The symbol $\N$ will denote the
set of nonnegative integers and $\Q$, the field of rational numbers.

We refer the reader to \cite{DNR01} or \cite{sM93a} for the usual notation,
definitions and general facts on Hopf algebra theory. In particular, we shall
denote the comultiplication and counit maps of a Hopf algebra by $\Delta$ and
$\varepsilon$, respectively, and we shall adopt Sweedler's notation 
$\Delta(h)=\sum_{(h)} h_{(1)}\otimes h_{(2)}$, for an element $h$ in a Hopf algebra.

\section{Graded module algebras}

Given an algebra $H$ and an $H$-module $W$, we let
$$
\rho_W\colon H\longrightarrow\End W
$$
denote the algebra map that affords the $H$-module structure on $W$,
that is, $\rho_W$ is defined by
$$
\rho_W(h)(w) = h\cdot w,\quad\text{for all $h\in H$ and $w\in W$.}
$$
If $W$ is finite-dimensional, denote by $\chi_W$ the character of
$W$, that is, $\chi_W\colon H\rightarrow \field$ is the linear
map given by
$$
\chi_W(h) = \trace(\rho_W(h)), \quad\text{for all $h\in H$,}
$$
where $\trace$ stands for the trace of a linear endomorphism.

Now suppose that $H$ is a Hopf algebra. In this case,
$W^H=\{w\in W : h\cdot w = \varepsilon(h)w\text{, for all $h\in H$}\}$
is an $H$-submodule of $W$.

We shall make use of characters of modules over a semisimple
Hopf algebra. For that, we recall some facts on semisimple Hopf algebras. 
(See \cite{DNR01} for proofs.) A Hopf
algebra $H$ is said to be semisimple if it is semisimple as
an algebra. An element $t$ of a Hopf algebra $H$ is called a left 
integral in $H$ if $ht=\varepsilon(h)t$, for all $h\in H$.
Semisimple Hopf algebras can be characterized in terms of
integrals via the following version of Maschke's theorem:
a Hopf algebra $H$ is semisimple if and only if there exist
a left integral $t$ in $H$ with $\varepsilon(t)=1$. Moreover,
semisimple Hopf algebras are finite-dimensional. Because
left integrals with counit $1$ in a semisimple Hopf algebra
are central idempotent elements, we have the following (easily proved)
well-known facts.

\begin{proposition}\label{prop:1}
Let $H$ be a semisimple Hopf algebra, let $t\in H$ be a left integral
with $\varepsilon(t)=1$ and let $W$ be a finite-dimensional $H$-module. Then $\rho_W(t)$
is an $H$-module endomorphism of $W$ satisfying $\rho_W(t)^2=\rho_W(t)$
and $\imag\rho_W(t)=W^H$. In particular, if $\charac\field=0$, then
$\dim W^H = \chi_W(t)$.\qed
\end{proposition}


\medskip

We now specialize to Hopf algebra actions on graded algebras.

An algebra $A$ is said to be graded if there exist subspaces $A_n$, for $n\in\N$,
satisfying
$$
A=\bigoplus_{n\in\N}A_n\quad\text{and}\quad 
A_mA_n\subseteq A_{m+n}\text{, for all $m,n\in\N$.}
$$
When $\dim A_n$ is finite, for all $n\in\N$, we associate to the
grading of $A$ a formal power series with integer coefficients, called the
Hilbert (or Poincar\'e) series of $A$, defined by
$$
P(A,z) = \sum_{n\in\N}(\dim A_n)z^n\in \Z[[z]].
$$

Now let $H$ be a Hopf algebra and suppose that the graded algebra $A$ is
an $H$-module algebra. We say that the action of $H$ on $A$ is homogeneous
with respect to the grading if $H\cdot A_n\subseteq A_n$, for all $n\in\N$,
or, equivalently, if $A_n$ is an $H$-submodule of $A$, for all $n\in\N$.
In this case, $A^H$ is a subalgebra of $A$ which inherits the grading from $A$: 
$$
A^H = \bigoplus_{n\in\N} (A_n\cap A^H).
$$
And, hence, it makes sense to speak of the Hilbert series 
$P(A^H,z)$ of $A^H$; since $A_n\cap A^H = A_n^H$, it follows
that $P(A^H,z)=\sum_{n\in\N}(\dim A_n^H)z^n$. We shall make use of
Proposition~\ref{prop:1} in order to determine the coefficients of
this series. In the above context we shall denote the character
of the $H$-module $A_n$ by, simply, $\chi_n$. So, applying 
Proposition~\ref{prop:1} to the $H$-modules $A_n$, we obtain the
following.

\begin{corollary}\label{cor:2}
Let $H$ be a semisimple Hopf algebra and let $t$ be a left integral
in $H$ with $\varepsilon(t)=1$. Let $A=\bigoplus_{n\in\N}A_n$ be a graded
$H$-module algebra with a homogeneous action of $H$ and suppose that
$\dim A_n$ is finite for all $n\in\N$. Then
$P(A^H,z) =\sum_{n\in\N}\chi_n(t)z^n$.\qed
\end{corollary}

\section{Linear actions}

In this section, we investigate the Hilbert series of invariants
of tensor algebras under homogeneous actions of a semisimple Hopf algebra.

Let $H$ be Hopf algebra and let $V$ be an $H$-module. Let $T(V)_0=\field$, viewed
as the trivial $H$-module, and for $n\in\N, n\geq 1$, let $T(V)_n=V\otimes T(V)_{n-1}$. 
The $H$-module structure on $V$ induces an $H$-module structure on $T(V)_n$ such
that
$$
h\cdot(v_1\otimes\dots\otimes v_n) = 
\sum_{(h)}(h_{(1)}\cdot v_1)\otimes\dots\otimes(h_{(n)}\cdot v_n),
$$
for $h\in H$ and $v_i\in V$ for $i=1,\dots,n$.

Now let
$$
T(V) = \bigoplus_{n\in\N} T(V)_n
$$
be the tensor algebra of $V$. It is an $H$-module; in fact, $T(V)$
is an $H$-module algebra. When the tensor algebra of a vector
space $V$ has an $H$-module algebra structure which is induced by 
an $H$-module structure on $V$, we shall call the action of $H$ on
$T(V)$ linear.

The tensor algebra $T(V)$ is, by definition, graded by the
tensor powers of $V$ and a linear action of a Hopf algebra
$H$ on $T(V)$ is homogeneous. Therefore, we can apply the the
ideas of the previous section to this case and look into
the Hilbert series of the subalgebra of invariants $T(V)^H =
\bigoplus_{n\in\N} T(V)_n^H$ of $T(V)$ under the action of $H$.

\begin{remark}
Once a basis $X$ of $V$ is fixed, we have a natural isomorphism between 
$T(V)$ and the free associative algebra
$\fa{\field}{X}$ on $X$ over $k$. So if $\dim V= d$, then $T(V)$ is a free
algebra of rank $d$ over $k$. Under this viewpoint a linear action
of a Hopf algebra $H$ on $T(V)$ is just an $H$-module algebra structure on
$\fa{\field}{X}$ such that for each free generator $x\in X$ and each $h\in H$,
there exist scalars $\lambda_y$, $y\in X$, a finite number of which nonzero,
such that $h\cdot x = \sum_{y\in X}\lambda_yy$.
\end{remark}

We shall present conditions under which the Hilbert series of
$T(V)^H$ is a rational function. 

Let $H$ be a Hopf algebra, let $V$ be a finite-dimensional vector space and suppose that 
$T(V)$ is an $H$-module algebra with a linear action of $H$. For each
$h\in H$, we shall denote by $\alpha_h(z)\in\field[[z]]$ the
power series defined by
$$
\alpha_h(z) = \sum_{n\in\N}\chi_n(h)z^n,
$$
where $\chi_n$ denotes the character of the $H$-module $T(V)_n$.

\begin{lemma} \label{le:3}
Let $H$ be a finite dimensional Hopf algebra with basis $\{h_1,\dots,h_r\}$.
For each $i=1,\dots,r$, denote $\alpha_i(z) = \alpha_{h_i}(z)$. 
Then $\alpha_i(z)\in \field(z)$, for every $i=1,\dots,r$.
\end{lemma}

\begin{proof}
For each $i=1,\dots,r$, there exist uniquely determined $a_{ij}\in H, j=1,\dots,n$
such that
$$
\Delta(h_i) = \sum_{j=1}^r a_{ij}\otimes h_j.
$$ 
Let $\lambda_{ij} = \chi_1(a_{ij})\in\field$.
So $\chi_0(h_i)=\varepsilon(h_i)$ and for $n\in\N$, $n\geq 1$, we have
$$
\chi_n(h_i) = \sum_{j=1}^r\chi_1(a_{ij})\chi_{n-1}(h_j)
  =\sum_{j=1}^r\lambda_{ij}\chi_{n-1}(h_j).
$$
This is because, being characters, the $\chi_n$ are multiplicative, in the
sense that, for all $h\in H$, $\chi_n(h)=\sum_{(h)}\chi_1(h_{(1)})\chi_{n-1}(h_{(2)})$, for,
by definition, $T(V)_n=T(V)_1\otimes T(V)_{n-1}$ and the action of $H$ on $T(V)_n$ is induced
by its action on $T(V)_1$ and on $T(V)_{n-1}$ through $\Delta$.

Therefore,
\begin{equation*}\begin{split}
\alpha_i(z)&= \sum_{n\in\N}\chi_n(h_i)z^n  = \varepsilon(h_i) + \sum_{n\geq 1}\chi_n(h_i)z^n\\
  &= \varepsilon(h_i)+\sum_{n\geq 1}\sum_{j=1}^r\lambda_{ij}\chi_{n-1}(h_j)z^n\\
  &= \varepsilon(h_i)+\sum_{j=1}^r\lambda_{ij}z\sum_{n\geq 1}\chi_{n-1}(h_j)z^{n-1}\\
  &= \varepsilon(h_i)+\sum_{j=1}^r\lambda_{ij}z\alpha_j(z).
\end{split}\end{equation*}

It follows that, for each $i=1,\dots,r$, we have
$$
(\lambda_{ii}z-1)\alpha_i(z)+\sum_{j\ne i}\lambda_{ij}z\alpha_j(x) = -\varepsilon(h_i).
$$
That is, $(\alpha_1(z),\dots,\alpha_r(z))$ is a solution of the linear system
$M(z)X=\eta$ over $\field(z)$, where
$$ M(z)=
\begin{bmatrix}
\lambda_{11}z-1 & \lambda_{12}z   & \cdots & \lambda_{1r}z  \\
\lambda_{21}z   & \lambda_{22}z-1 & \cdots & \lambda_{2r}z  \\
\vdots          & \vdots          & \ddots & \vdots         \\
\lambda_{r1}    & \lambda_{r2}z   & \cdots & \lambda_{rr}z-1
\end{bmatrix}
\quad
\text{and}
\quad
\eta = 
-\begin{bmatrix}
\varepsilon(h_1)\\\varepsilon(h_2)\\\vdots\\\varepsilon(h_r)
\end{bmatrix}.
$$
Now $\det M(0)=(-1)^r$ and, hence, $M(z)$ is an invertible matrix
over $\field(z)$. It follows, by Cramer's rule, that $\alpha_i(z)\in \field(z)$,
for all $i=1,\dots,r$.
\end{proof}

We point out that the proof of Lemma~\ref{le:3} not only does give
the rationality of the series $\alpha_i(z)$, but also provides
explicit formulas for them as quotients of polynomials with coefficients in $\field$.
We shall explore this fact below.

\begin{theorem} \label{th:4}
Let $H$ be a semisimple Hopf algebra, let $V$ be a finite dimensional-vector space
and suppose that the tensor algebra $T(V)$ of $V$ is an $H$-module algebra with
a linear action of $H$. If $\charac\field=0$, then $P(T(V)^H,z)\in\Q(z)$.
\end{theorem}

\begin{proof}
Let $\{h_1,\dots,h_r\}$ be basis for $H$, let $t$ be a left integral in $H$ 
with $\varepsilon(t)=1$ and let $\mu_i\in\field$, $i=1,\dots,r$ be scalars
such that $t=\sum_{i=1}^r \mu_ih_j$. Then, by Corollary~\ref{cor:2}, we
get $P(T(V)^H,z)=\sum_{n\in\N}\chi_n(t)z^n=\sum_{i=1}^r\mu_i\alpha_i(z)$, 
where $\alpha_i(z)=\sum_{n\in\N}\chi_n(h_i)z^n$. It follows from Lemma~\ref{le:3}
that $P(T(V)^H,z)\in\field(z)$. Now, since $P(T(V)^H,z)$ has rational
coefficients and is a rational function over $\field\supseteq\Q$, it follows
from Kronecker's theorem on Hankel operators (see, e.g., \cite[Th.~3.11]{jP88}) that
$P(T(V)^H,z)\in\Q(z)$.
\end{proof}

If $G$ is a finite group and we let $H=\field G$ be the group algebra
of $G$ over $\field$ with its usual Hopf algebra structure, then
the proofs of Theorem~\ref{th:4} and Lemma~\ref{le:3}, using $G$ as
a $\field$-basis for $H$, allow us to
reconstruct Dicks and Formanek's formula for the the subalgebra of
invariants of a free algebra under the action of a finite group of
linear automorphisms when $\charac\field=0$ (see \cite{DF82}):
$$
P(T(V)^H) = \frac{1}{|G|}\sum_{g\in G}\frac{1}{1-\chi_1(g)z}.
$$

In \cite{FM10}, the authors have also explored the particular case of
$H=(\field G)^{\ast}$, the dual Hopf algebra of the group algebra $\field G$ 
of a finite group $G$. In this case, a linear $H$-module algebra 
structure on $T(V)$ amounts to a ``linear'' grading on $T(V)$ by
the group $G$. It has been possible to obtain a sufficiently explicit formula 
for the Hilbert series of the homogeneous component associated to the
identity of $G$ in order to produce a criterion for finite generation
of this subalgebra.

\end{document}